\documentclass[reqno]{amsart}

\usepackage{amsmath,amssymb,amsthm,mathrsfs}
\usepackage[colorlinks,linkcolor=blue,citecolor=blue]{hyperref}
\usepackage{xcolor}
\theoremstyle{plain}
\numberwithin{equation}{section}

\newcommand{\BC}{{\mathbb C}}

\newcommand{\BK}{{\mathbb K}}
\newcommand{\BN}{{\mathbb N}}

\newcommand{\BR}{{\mathbb R}}

\newcommand{\cD}{{\mathcal D}}

\newcommand{\cH}{{\mathcal H}}

\newcommand{\cX}{{\mathcal X}}
\newcommand{\cY}{{\mathcal Y}}
\newcommand{\de}{\delta}\newcommand{\De}{\Delta}
\newcommand{\vep}{\varepsilon}

\newcommand{\im}{\textup{Im\,}}
\newcommand{\re}{\textup{Re\,}}

\newcommand{\mat}[2]{\ensuremath{\left[\begin{array}{#1}#2\end{array} \right]}}

\newcommand{\sbm}[1]{\left[\begin{smallmatrix} #1\end{smallmatrix}\right]}

\newcommand{\tu}[1]{\textup{#1}}
\newcommand{\wtil}[1]{{\widetilde{#1}}}

\newcommand{\half}{\frac{1}{2}}
\newcommand{\ands}{\quad\mbox{and}\quad}

\theoremstyle{plain}
\newtheorem{theorem}{Theorem}[section]
\newtheorem{corollary}[theorem]{Corollary}
\newtheorem{lemma}[theorem]{Lemma}
\newtheorem{proposition}[theorem]{Proposition}
\theoremstyle{definition}

\newtheorem{example}[theorem]{Example}
\newtheorem{remark}[theorem]{Remark}

\newcommand{\mattwo}[4]{\begin{bmatrix}#1 & #2\\ #3 & #4\end{bmatrix}}

\begin{document}

\title{Cauchy-Riemann equations for free noncommutative functions}

\author[S. ter Horst]{S. ter Horst}
\address{S. ter Horst, Department of Mathematics, Unit for BMI, North-West
University, Potchefstroom, 2531 South Africa, and DST-NRF Centre of Excellence in Mathematical and Statistical Sciences (CoE-MaSS)}
\email{Sanne.TerHorst@nwu.ac.za}

\author[E.M. Klem]{E.M. Klem}
\address{E.M. Klem, Department of Mathematics, Unit for BMI, North-West University, Potchefstroom, 2531 South Africa}
\email{estiaanklem@gmail.com}

\thanks{This work is based on the research supported in part by the National Research Foundation of South Africa (Grant Numbers 90670, 118583, and 94069).}

\subjclass[2010]{Primary 32A10; Secondary 46L52, 26B05}

%32A10 Holomorphic functions
%46L52 Noncommutative function spaces
%26B05 Continuity and differentiation questions

\keywords{Cauchy-Riemann equations, free noncommutative functions, real noncommutative functions}

\begin{abstract}
In classical complex analysis analyticity of a complex function $f$ is equivalent to differentiability of its real and imaginary parts $u$ and $v$, respectively, together with the Cauchy-Riemann equations for the partial derivatives of $u$ and $v$. We extend this result to the context of free noncommutative functions on tuples of matrices of arbitrary size. In this context, the real and imaginary parts become so called real noncommutative functions, as appeared recently in the context of L\"owner's theorem in several noncommutative variables. Additionally, as part of our investigation of real noncommutative functions, we show that real noncommutative functions are in fact noncommutative functions.
\end{abstract}

\maketitle

%%%%%%%%%%%%%%%%%%%%%%%%%%%%%%%%%%%%%%%%%%%%%%%%%%
%%%%%%%%%%%%%%%%%%%%%%%%%%%%%%%%%%%%%%%%%%%%%%%%%%
\section{Introduction}\label{S:Intro}

Over the last decade a theory of free noncommutative (nc) functions that are evaluated in tuples of matrices of arbitrary size was developed. The theory becomes particularly rich when the functions have a domain that is assumed to be right (or left) admissible, in which case the functions admit a Taylor expansion and, under mild boundedness assumptions, are analytic. We refer to \cite{KVV14} for the first book that presents a comprehensive account of the theory, as well the seminal paper \cite{T73} by J.L. Taylor. Precise definitions will be given a little further in this introduction.

More recently, in connection with L\"owner's theorem \cite{PTD17,P18,P-Arx16}, the notion of real nc functions appeared. These functions have domains that consist of tuples on Hermitian matrices, precluding the right (or left) admissibility property, and satisfy slightly different conditions. Another instance where real nc function come up in a natural way is as the real and imaginary part of a nc function. In the present paper we derive the noncommutative Cauchy-Riemann equations for the real and imaginary part of a nc function and consider the question when two real nc functions satisfying the noncommutative Cauchy-Riemann equations appear as the real and imaginary part of a nc function.

We will now provide more precise definitions and state our main result. Throughout $\BC^{n\times n}$ denotes the complex vector space of $n\times n$ complex matrices and $\cH_n$ the real vector space of $n\times n$ Hermitian matrices. For a positive integer $d$, we consider functions with domains in
\[
\BC^d_\tu{nc}:=\coprod_{n=1}^\infty (\BC^d)^{n\times n}=\coprod_{n=1}^\infty (\BC^{n\times n})^d \quad\mbox{or}\quad
\cH_\tu{nc}^d:=\coprod_{n=1}^\infty (\cH_n)^d.
\]
In case $d=1$ we omit it as a superscript and simply write $\BC_\tu{nc}$ and $\cH_\tu{nc}$. A subset $\cD$ of $\BC^d_\tu{nc}$ or $\cH_\tu{nc}^d$ is said to be a nc set in case it {\em respects direct sums}:
\[
X,Y\in\cD \quad \Longrightarrow \quad X\oplus Y=\mat{cc}{X&0\\0&Y}\in\cD.
\]
In some papers the converse implication as well as additional features are also assumed, cf., \cite{PTD17,P18}. See Lemma \ref{L:realfreesets} below as well as the paragraph preceding this lemma. For a nc set $\cD$ and a positive integer $n$ we define $\cD_n:=\cD\cap \BC^{n\times n}$. A nc set $\cD\subset \BC^d_\tu{nc}$ is called {\em right admissible} in case
\begin{equation}\label{rightadmis}
X\in\cD_n,\, Y\in\cD_m,\, Z\in\BC^{n\times m}  \Longrightarrow
\mat{cc}{X&\! rZ\\ 0&\! Y}\in\cD_{n+m}\mbox{ for some $0\neq r\in\BC$}.
\end{equation}
In case the nc set $\cD$ is right admissible and closed under similarity, then the ``for some'' part in the right-hand side of \eqref{rightadmis} can be replaced by ``for all.''  There is a dual notion of left admissibility, see page 18 and onwards in \cite{KVV14}, but we will not need this notion in the present paper.

A function $w:\cD \to \BC_\tu{nc}$ whose domain $\cD$ is a nc set in $\BC_\tu{nc}^d$ is called a {\em nc function} in case it has the following properties:
\begin{itemize}\label{NC123}
		\item[(NC-i)] $w$ is {\em graded}, i.e., $w(\cD_n)\subset \BC^{n\times n}$ for $n=1,2,\ldots$;

		\item[(NC-ii)]  $w$ {\em respects direct sums}, i.e., for all $X,Y\in\cD$ we have
\[
w(X\oplus Y)=w(X)\oplus w(Y);
\]

		\item[(NC-iii)] $w$ {\em respects similarities}, i.e., for all $X\in\cD_n$, $S\in\BC^{n\times n}$ invertible so that $S X S^{-1}\in\cD_n$, we have
\[
w(S X S^{-1})=S w(X) S^{-1}.
\]
\end{itemize}
Much of the theory of nc functions developed in \cite{KVV14} is for nc functions whose domains are right (or left) admissible, in which case for each $X$, $Y$, $Z$ and $r\neq 0$ as in \eqref{rightadmis} one can define the right difference-differential operator $\Delta w(X,Y)$ at the point $Z$ via
\begin{equation}\label{DDoperator}
w\left(\mat{cc}{X& r Z\\ 0&Y}\right) = \mat{cc}{w(X)& r \Delta w(X,Y)(Z)\\ 0&w(Y)},
\end{equation}
with the zero and two block diagonal entries following from (NCi)--(NCiii). This right difference-differential operator is linear in $Z$ and provides a difference formula for $w$ leading to the so-called Taylor-Taylor expansion of $w$, and, under certain boundedness assumptions on $w$, provides the  G\^{a}teaux-derivative of $w$; see \cite{KVV14} for an elaborate treatment. Recall that the G\^{a}teaux- or G-derivative of a function $g:\cD_g\to\cY$ with domain $\cD_g\subset\cX$, with $\cX$ and $\cY$ Banach spaces over the field $\BK=\BC$ or $\BK=\BR$, at a point $X\in\cX$ in $\cD_g$ in the direction $Z\in\cX$ is given by
\begin{equation}\label{Gderive}
Dg(X)(Z):=\lim_{\BK\ni t\to 0}\frac{g(X+tZ)-g(X)}{t},
\end{equation}
provided the limit exist. Then $g$ is said to be G\^{a}teaux- or G-differentiable in case $\cD_g$ is open and $Dg(X)(Z)$ exists for all $X\in\cD_g$ and all $Z\in\cX$. In the case of nc functions, G-differentiability means that for each positive integer $n$ the restriction of the domain to $(\BC^{n\times n})^d$ should be G-differentiable; see Section \ref{S:DirDer} for further details and references on G-differentiability as well as Fr\'{e}chet- or F-differentiability.

A function $w:\cD_w \to \cH_\tu{nc}$ is called a {\em real nc function} in case its domain $\cD$ is a nc set contained in $\cH_\tu{nc}^d$  which is graded and respects direct sums, i.e., (NC-i) and (NC-ii) above hold, and
\begin{itemize}\label{RNC3}
		\item[(RNC-iii)] $w$ {\em respects unitary equivalence}, i.e., for all $X\in\cD_n$, $U\in\BC^{n\times n}$ unitary so that $U X U^*\in\cD_n$, we have
\[
w(U X U^*)=U w(X) U^*.
\]
\end{itemize}
Despite the seeming limitation of unitary equivalence over similarity, one of the contributions of the present paper is the observation that real nc functions are also nc functions, see Theorem \ref{T:realNC=NC} below. Hence (NC-i), (NC-ii) and (RNC-iii) imply (NC-iii). This result relies heavily on the fact that the domains of real nc functions consist of tuples of Hermitian matrices only. The latter also implies that the domains of real nc functions are `nowhere right admissible,' and hence much of the theory developed in \cite{KVV14} does not apply to real nc functions.

Now, given an nc function $f$ on a right admissible domain $\cD_f\subset \BC^d_\tu{nc}$, we write
\begin{equation}\label{fdec}
f(A+iB)=u(A,B)+iv(A,B),\quad A+iB\in\cD_f,
\end{equation}
for $A,B\in\cH_\tu{nc}^d$ of the same size and with
\[
u(A,B):=\re f(A+iB)\ands v(A,B):=\im f(A+iB).
\]
This defines real nc functions $u$ and $v$ on domain
$\cD=\{(A,B)\in\cH_\tu{nc}^{2d} \colon A+iB \in\cD_f\}$, which is open in $\cH_\tu{nc}^{2d}$ precisely when $\cD_f$ is open in $\BC_\tu{nc}^d$; in both cases open means that the restriction of the domain to $n\times n$ matrices is open in $\cH_n^{2d}$ and $(\BC^{n\times n})^d$, respectively. Furthermore, in case $f$ is G-differentiable, then so are $u$ and $v$ and their G-derivatives satisfy the following noncommutative Cauchy-Riemann equations
\begin{equation}\label{CR-intro}
Du(A,B)(Z_1,Z_2)=Dv(A,B)(-Z_2,Z_1),\ (A,B)\in \cD_n,\ Z_1,Z_2\in\cH_n,\ n\in\BN.
\end{equation}
See Theorem \ref{T:RCpart} for these claims as well as additional results.

Conversely, one may wonder whether G-differentiable real nc functions $u$ and $v$ with open domains $\cD_u$ and $\cD_v$, respectively, in $\cH_\tu{nc}^{2d}$ that satisfy \eqref{CR-intro} on $\cD=\cD_u\cap\cD_v$ define a nc function f via \eqref{fdec}. For this purpose, G-differentiability does not seem to be the appropriate notion of differentiability, and we will rather assume the stronger notion of F-differentiability, in which case the derivative is still obtained via \eqref{Gderive}; see Section \ref{S:DirDer} for further details. Even in classical complex analysis this phenomenon occurs, see \cite{DGM75,GM78} as well as Remark \ref{R:G-diff} below. Our main result is the following theorem.

\begin{theorem}\label{T:CR}
Let $u$ and $v$ be real nc functions with open domains $\cD_u$ and $\cD_v$, respectively, in $\cH_\tu{nc}^{2d}$ that are F-differentiable and satisfy the nc Cauchy-Riemann equations \eqref{CR-intro} on $\cD=\cD_u\cap\cD_v$. Define $f$ on $\cD_f=\{A+iB\in\BC_\tu{nc}^d\colon (A,B)\in\cD\}$ via \eqref{fdec}. Then $f$ is a F-differentiable nc function.
\end{theorem}

Apart from the present introduction, this paper consists of four sections. In Section \ref{S:real nc} we prove that real nc functions are nc function, consider some examples and look at domain extensions. Next, in Section \ref{S:DirDer} we review the notions of G\^ateaux- and Fr\'echet differentiability for nc functions. The domains of real nc functions are not right-admissible so that the G-derivative cannot be determined algebraically through the difference-differential operator. In the following section we derive properties of the real and imaginary parts of an nc function, including the nc Cauchy-Riemann equations. Finally, in Section \ref{S:CR} we consider the converse direction and prove Theorem \ref{T:CR}.

%%%%%%%%%%%%%%%%%%%%%%%%%%%%%%%%%%%%%%%%%%%%%%%%%%
%%%%%%%%%%%%%%%%%%%%%%%%%%%%%%%%%%%%%%%%%%%%%%%%%%
\section{Real nc functions are nc functions}\label{S:real nc}

In this section we focus on real nc functions only, without assuming any form of differentiability. Our main result is the following theorem.

\begin{theorem}\label{T:realNC=NC}
Real nc functions are nc functions.
\end{theorem}

In order to prove this result we first show that real nc functions also respect intertwining.

\begin{proposition}\label{P:RealIP}
A graded function $w:\cD\to \cH_\tu{nc}$ on a nc set $\cD\subset\cH_\tu{nc}$ respects direct sums and unitary equivalence if and only if it respects intertwining: if $X\in\cD_n$, $Y\in\cD_m$, and $T\in\BC^{n\times m}$ so that $XT=TY$, then $w(X)T=Tw(Y)$.
\end{proposition}

\begin{proof}[\bf{Proof}]
The necessity follows from Proposition 2.1 in \cite{KVV14}. Assume $w$ respects direct sums and unitary equivalence, i.e., $w$ is a real nc function. Let $X\in\cD_n$, $Y\in\cD_m$, and $T_0\in\BC^{n\times m}$ so that $XT_0=T_0 Y$. If $T_0=0$, then it is trivial that $w(X)T_0=T_0w(Y)$, so assume $T_0\neq0$. Set $T=\|T_0\|^{-1} T_0$ so that $\|T\|=1$. Let $D_{T}:=(I-T^\ast T)^{1/2}$ and $D_{T^*}:=(I-T T^\ast)^{1/2}$ be the defect matrices of the contractions $T$ and $T^*$, respectively. Since $X$ and $Y$ are Hermitian we have
\[
T^*X=YT^*.
\]
Therefore
\[
X D_{T^\ast}^2=X(I-T T^\ast)
	=X-T Y T^\ast
	=X-T T^\ast X
	=(I-T T^\ast)X
	=D_{T^\ast}^2 X,
\]
and similarly $Y D_T^2=D_T^2 Y$. By the spectral theorem we have $X D_{T^\ast}=D_{T^\ast} X$ and $Y D_{T}=D_{T} Y$. Let $U_T$ be the unitary rotation matrix associated with $T$:
\[
U_T=\mat{cc}{T& D_{T^\ast}\\D_T&-T^\ast}.
\]
Then
\begin{align*}
\mat{cc}{X&0\\0&Y} U_T
=\mat{cc}{XT& XD_{T^\ast}\\YD_T&-YT^\ast}
=\mat{cc}{TY& D_{T^\ast}X\\D_TY&-T^\ast X}
= U_T \mat{cc}{Y&0\\0&X}.
\end{align*}
Hence
\[
U_T^*\mat{cc}{X&0\\0&Y} U_T=\mat{cc}{Y&0\\0&X}\in\cD_{n+m}.
\]
Since $w$ respects direct sums and unitary similarities, we have that
\begin{align*}
\mat{cc}{w(Y)&0\\0&w(X)}
&=w\left(\mattwo{Y}{0}{0}{X}\right)
=w\left(U_T^\ast\mattwo{X}{0}{0}{Y}U_T\right)\\
&=U_T^\ast w\left(\mattwo{X}{0}{0}{Y}\right) U_T =U_T^\ast\mattwo{w(X)}{0}{0}{w(Y)}U_T.
\end{align*}
This shows that
\begin{align*}
\mat{cc}{w(X)T& w(X) D_{T^\ast}\\w(Y)D_T&-w(Y)T^\ast}
&=\mattwo{w(X)}{0}{0}{w(Y)}U_T=U_T\mattwo{w(Y)}{0}{0}{w(X)}=\\
&=\mat{cc}{Tw(Y)& D_{T^\ast}w(X)\\D_T w(Y)&-T^\ast w(X)}.
\end{align*}
Comparing the left-upper corners in the above identity yields $w(X)T=Tw(Y)$, and thus
\[
w(X)T_0=\|T_0\|w(X)T=\|T_0\|Tw(Y)=T_0 w(Y)
\]
as desired.
\end{proof}

\begin{proof}[\bf Proof of Theorem \ref{T:realNC=NC}]
This is now straightforward. By assumption $w$ is graded and respects direct sums. Let $X\in\cD_n$ and $T\in\BC^{n\times n}$ invertible so that $Y:=T^{-1}X T\in\cD_n$. Then $XT=TY$, and thus $w(X)T=Tw(Y)$ holds by Proposition \ref{P:RealIP}. Therefore, we have
$w(T^{-1}XT)=T^{-1}Tw(Y)=T^{-1} w(X)T$.
\end{proof}

\begin{remark}\label{R:iii'}
Theorem \ref{T:realNC=NC} shows that assumptions (NC-i),(NC-ii) and (RNC-iii) imply (NC-iii), that is:\ For $X\in\cD_{n}$, $S\in\BC^{n\times n}$ invertible so that $S X S^{-1}\in\cD_n$, we have
\[
w(S X S^{-1})=S w(X) S^{-1}.
\]
An important feature here is that $Y:=S X S^{-1}\in\cD_n$ implies, in particular, that $Y$ is Hermitian. In this case, by \cite[Problem 4.1.P3]{HJ13}, $X$ and $Y$ are not only similar, but also unitarily equivalent. In fact, we have $Y=U X U^*$, where $U$ is the unitary matrix from the polar decomposition of $S$. Consequently, we have $w(S X S^{-1})=w(U X U^*)= U w(X) U^*$. However, to arrive at $w(S X S^{-1})= S w(X) S^{-1}$ it still seems necessary to have a result like Proposition \ref{P:RealIP}, at least for the case of positive definite similarities.
\end{remark}

\begin{example}\label{Ex:rnc-simple}
It also follows from Theorem \ref{T:realNC=NC} that real nc functions are only distinguishable from other nc functions by the fact that their domains are contained in $\cH_\tu{nc}^d$ for some positive integer $d$. Simple examples show that the assumption $\cD\subset \cH_\tu{nc}^d$ cannot be removed without Theorem \ref{T:realNC=NC} losing its validity. Any one of the functions
\[
w_1(X)=X^*,\quad w_2(X)=(X^*X)^{\half}
\]
can be defined on $\BC_\tu{nc}$, where they satisfy (NC-i),(NC-ii) and (RNC-iii) but not (NC-iii), hence they are not nc functions on $\BC_\tu{nc}$, but their restrictions to $\cH_\tu{nc}$ are, by Theorem \ref{T:realNC=NC}.
\end{example}

\begin{example}\label{Ex:rnc}
For $d=1$ more intricate examples can easily be constructed. Via the continuous functional calculus, any continuous function $w$ with domain in $\BR$ can be extended to a real nc function on the nc set of Hermitian matrices whose spectrum is contained in the domain of $w$, even when it is not differentiable. Clearly the resulting real nc function is also not differentiable in case $w$ is not.

It is not directly clear how a continuous function of several real variables can be extended to a real nc function, except when the domain is restricted to tuples of commuting matrices. In passing, we note that a (unintentional) non-example is given in \cite{JS18}, where an extension of a function in several real variables to a noncommutative domain is considered, which, after some minor modifications, can be restricted to a nc domain in $\cH_\tu{nc}$, leading to a non-graded function (it maps $\cH_n^d$ to $\cH_{n d}$) which does satisfy conditions (NC-ii) and (NC-iii).
\end{example}

\paragraph{\bf Domain extensions}
Since a real nc function $w$ with domain $\cD$ is a nc function, it follows from Proposition A.3 in \cite{KVV14} that $w$ can be uniquely extended to a nc function, also denoted by $w$, on the similarity invariant envelop of $\cD$:
\[
\cD^{\tu{(si)}}:=\{SXS^{-1} \colon X\in \cD_{n},\ S\in\BC^{n\times n} \mbox{ invertible}\}
\]
via
\[
w(Y)=w(SXS^{-1}):=S w(X) S^{-1}\quad (Y=SXS^{-1}\in \cD^{\tu{(si)}}).
\]
However, in general, $\cD^{\tu{(si)}}$ will not be contained in $\cH_\tu{nc}$, although all matrices in $\cD^{\tu{(si)}}$ have real spectrum only and the only nilpotent matrix in $\cD^{\tu{(si)}}$ is the zero matrix $0$, assuming $0\in\cD$. In the context of real nc functions it may be more natural to consider the extension of $w$ to the unitary equivalence invariant envelop
\[
\cD^{\tu{(ue)}}:=\{UXU^* \colon X\in \cD_{n},\ U\in\BC^{n\times n} \mbox{ unitary}\}=\cD^{\tu{(si)}}\cap \cH_\tu{nc},
\]
with $w$ extended as before. The fact that $\cD^{\tu{(ue)}}=\cD^{\tu{(si)}}\cap \cH_\tu{nc}$ again follows by \cite[Problem 4.1.P3]{HJ13}. As this is just the restriction to $\cD^{\tu{(ue)}}$ of the extension of $w$ to $\cD^{\tu{(si)}}$, clearly we end up with a real nc function extension of $w$ to $\cD^{\tu{(ue)}}$ which is uniquely determined by $w$.

In \cite{PTD17,P18} real free sets (restricted to the case where tensoring is done with the real topological vector space $\mathcal{R}=\BR^d$) are nc sets $\cD\subset \cH_\tu{nc}$ that are closed under unitary equivalence and have the following property:
\begin{itemize}
\item[(a)] For $X,Y\in\cH_\tu{nc}$ we have $X,Y\in \cD$ if and only of $X\oplus Y\in\cD$.
\end{itemize}
One implication is true by the assumption that $\cD$ is a nc set, but the other direction need not be true for the unitary equivalence envelop of a nc set contained in $\cH_\tu{nc}$.

\begin{lemma}\label{L:realfreesets}
Let $\cD\subset\cH_\tu{nc}$ be a nc set. Then the unitary equivalence envelop $\cD^\tu{(ue)}$ of $\cD$ is a real free set if and only if it is closed under injective intertwining:
\[
\mbox{If $X\in\cD^\tu{(ue)}_n$, $Y\in\cH_m$ and $S\in \BC^{n\times m}$ injective so that  $XS=SY$, then $Y\in\cD^\tu{(ue)}$.}
\]
\end{lemma}

\begin{proof}[\bf Proof]
Assume $\cD^\tu{(ue)}$ is closed under injective intertwining. Since $\cD$ is a nc set, so is $\cD^\tu{(ue)}$, by \cite[Proposition A.1]{KVV14}. Hence it remains to show that for $X,Y\in\cH_\tu{nc}$ with $X\oplus Y\in\cD$ also $X,Y\in\cD$. This follows by taking $S=S_1=\sbm{I\\0}$ and $S=S_2=\sbm{0\\I}$, respectively, with sizes compatible with the decomposition of $X\oplus Y$. Indeed, clearly $S_1$ and $S_2$ are injective and we have $(X\oplus Y)S_1=S_1 X$ and $(X\oplus Y)S_2=S_2 Y$. Thus $\cD^\tu{(ue)}$ is a nc set in $\cH_\tu{nc}$ which is closed under unitary equivalence and satisfies (a), hence it is a real free set.

For the converse direction, assume $\cD^\tu{(ue)}$ is  a real free set. Take $X\in\cD^\tu{(ue)}_n$, $Y\in\cH_m$ and $S\in \BC^{n\times m}$ injective so that  $XS=SY$. Since $\cD^\tu{(ue)}$ is closed under unitary equivalence and $S$ is injective, without loss of generality $S=\sbm{S_1\\ 0}$ with $S_1$ invertible. Then $XS=SY$ implies $\tu{Ran}(S)$ is invariant for $X$. However, $X$ is Hermitian, so that $\tu{Ran}(S)$ is in fact a reducing subspace for $X$. Hence $X=X_1\oplus X_2$ with respect to the same decomposition as for $S$. Then property (a) implies $X_1\in\cH_m$ is in $\cD^\tu{(ue)}$, and $XS=SY$ yields $X_1 S_1=S_1 Y$, i.e., $X_1 =S_1 Y S_1^{-1}$. Hence $X_1$ and $Y$ are similar. Since $X_1$ and $Y$ are Hermitian, $X_1$ and $Y$ are also  unitarily equivalent, by \cite[Problem 4.1.P3]{HJ13}. Hence $Y$ is in $\cD^\tu{(ue)}$.
\end{proof}

%\newpage

%%%%%%%%%%%%%%%%%%%%%%%%%%%%%%%%%%%%%%%%%%%%%%%%%%
%%%%%%%%%%%%%%%%%%%%%%%%%%%%%%%%%%%%%%%%%%%%%%%%%%
\section{Differentiability of nc functions}\label{S:DirDer}

For differentiation of vector-valued functions several notions exist, and these may differ for real and complex vector spaces. We refer to Section III.3 in \cite{HP57}, Section 5.3 in \cite{AH01} and Sections 2.3 and 2.4 in \cite{P13} for elaborate treatments, often at a much higher level of generality than required here. In this paper we will only encounter G\^{a}teaux (G-)differentiability and Fr\'{e}chet (F-)differentiability. In the context of nc functions over complex Banach spaces these notions are discussed in Chapter 7 of \cite{KVV14}, with a few remarks dedicated to the case of real Banach spaces. Here we will restrict to the case of nc functions on finite dimensional spaces, i.e., with domains in $\BC_\tu{nc}^d$ and $\cH_\tu{nc}^d$ with $d$ finite, as we do throughout the paper.

We start with the definitions of G-differentiability and F-differentiability, not distinguishing whether the field $\BK$ we
work over is $\BK=\BR$ or $\BK=\BC$, where in the case of $\BK=\BR$ we consider nc functions with domains contained in $\cH_\tu{nc}^d$ and for $\BK=\BC$ the nc functions are assumed to have a domain in $\BC_\tu{nc}^d$. Now let $w$ be a nc function defined on an open domain $\cD$ in $\cH_\tu{nc}^d$ (for $\BK=\BR$) or in $\BC_\tu{nc}^d$ (for $\BK=\BC$). Then for each $X\in\cD_{n}$ and $n\times n$ matrix $Z$ (in $\cH_n^d$ for $\BK=\BR$) we say $w$ is G-differentiable at $X$ in direction $Z$ in case the limit
\begin{equation}\label{G-diff}
Dw(X)(Z):=\lim_{\BK\ni t\to 0}\frac{w(X+tZ)-w(X)}{t}=\frac{d}{dt}w(X+tZ)\bigg|_{t=0}
\end{equation}
exists. In that case $Dw(X)(Z)$ is the G-derivative of $w$ at $X$ in direction $Z$. We say that $w$ is  G-differentiable in $X$ if $w$ is  G-differentiable at $X$ in each direction $Z$, and $w$ is called G-differentiable if it is G-differentiable in any $X\in\cD_w$. If $w$ is G-differentiable at $X\in\cD$, then the map $Z\mapsto Dw(X)(Z)$ is linear in $Z$. We shall usually refer to $Z$ as the {\em directional variable}.

Following \cite{HP57}, we say that the nc function $w$ is F-differentiable in $X\in\cD$ in case $w$ is G-differentiable in $X$ and the G-derivative $Dw(X)$ at $X$ satisfies
\begin{equation}\label{F-diff}
\lim_{\|Z\|\to 0} \frac{\|w(X+Z)-w(X)-Dw(X)(Z)\|}{\|Z\|} =0.
\end{equation}
Here, and in the sequel, the norm $\|Z\|$ for $Z=(Z^{(1)},\ldots,Z^{(d)})$ in $\cH_\tu{nc}^d$ or $\BC_\tu{nc}^d$ is given by $\|Z\|=\max_{k} \|Z^{(k)}\|$. Note that if for $X\in\cD$ there exists a homogeneous map $Z\mapsto Dw(X)(Z)$ that satisfies \eqref{F-diff}, then it must satisfy \eqref{G-diff}, so that $w$ is G-differentiable, and hence $D_w(X)(Z)$ is in fact linear in the directional variable $Z$. Hence, existence of a homogeneous map $Z\mapsto Dw(X)(Z)$ satisfying \eqref{F-diff} can be used as another definition of F-differentiability. Even in case $w$ is F-differentiable, we will refer to \eqref{G-diff} as the G-derivative of $w$.\medskip

\paragraph{\bf The case $\cD\subset \BC_\tu{nc}^d$ ($\BK=\BC$)} This case is discussed in detail in Chapter 7 of \cite{KVV14}. We just mention a few specific results relevant to the present paper and to illustrate the contrast with the case of real nc functions. Since the domain $\cD$ of $w$ is assumed to be open in $\BC_\tu{nc}^d$ it must be right-admissible and hence the difference-differential operator $\Delta w(X,Y)(Z)$ defined via \eqref{DDoperator} exists for all $X\in\cD_n$, $Y\in\cD_m$ and $Z\in (\BC^{n\times m})^d$.

By Theorem 7.2 in \cite{KVV14}, $w$ is G-differentiable in case $w$ is {\em locally bounded on slices}, that is, if for any $n$, $X\in \cD_n$ and any $Z\in\BC^{n\times n}$ there exists a $\vep>0$ so that $t\mapsto w(X+tZ)$ is bounded for $|t|<\vep$. Moreover, in that case we have $Dw(X)(Z)=\Delta w(X,X)(Z)$, and hence the G-derivative can be determined algebraically by evaluating $w$ in $\sbm{X& r Z\\ 0 & X}$ for small $r$. Furthermore, by Theorem 7.4 in \cite{KVV14}, $w$ is F-differentiable in case $w$ is {\em locally bounded}, that is, if for any $n$, $X\in \cD_n$ there exists a $\delta>0$ so that $w$ is bounded on the set of $Y\in \cD_n$ with $\|X-Y\|<\delta$. However, since we only consider the case of finite dimensional vector spaces, for $X\in\cD_n$ the linear map $Z\mapsto Dw(X)(Z)$ from $(\BC^{n\times n})^d$ to $\BC^{n\times n}$ is continuous, hence G-differentiability and F-differentiability coincide, by a result of Zorn \cite{Z46}.\medskip

\paragraph{\bf The case $\cD\subset \cH_\tu{nc}^d$ ($\BK=\BR$)}
The domain of real nc functions are `nowhere right admissible', hence one cannot in general define the difference-differential operator $\Delta w$ of a real nc function $w$ in the way it is done for nc functions defined on a right admissible nc set. Nonetheless, Proposition 2.5 in \cite{PTD17} provides a difference formula for real nc functions, provided they are F-differentiable.

As pointed out in Example \ref{Ex:rnc}, for $d=1$ any continuous function with domain in $\BR$ can be extended to a real nc function. Clearly G- or F-differentiability will not follow under local boundedness properties; consider, for instance, the function $w_2$ in Example \ref{Ex:rnc-simple}. The theory of G- and F-differentiability for functions between real Banach spaces is treated in Section 5.3 in \cite{AH01} and Sections 2.3 and 2.4 in \cite{P13}. It is not the case here that G- and F-differentiability coincide. By Proposition 5.3.4 in \cite{AH01} or Proposition 2.51 in \cite{P13}, a sufficient condition under which G-differentiability at a point $X\in\cD_n$ implies F-differentiability at $X$ is that the map $Y\mapsto Dw(Y)$ from $\cD_n$ into the space of linear operators from $\cH_n^d$ to $\cH_n$ is continuous at $X$. Even if $w$ is F-differentiable, there does not appear to be a general way to determine $Dw$ algebraically, since there is no difference-differential operator.

The formula presented in the next proposition can be seen as complementary to the difference formula in \cite[Proposition 2.5]{PTD17}.

\begin{proposition}\label{P:DirDerBlockForm}
Let $w:\cD \to\cH_\text{nc}$ be a G-differentiable real nc function on an open domain $\cD\subset \cH_\tu{nc}^d$. For $X\in\cD_n$ and $Z\in\cH_n^d$, with $n$ arbitrary, we have
\[
Dw\left(\mattwo{X}{0}{0}{X}\right)\left(\mattwo{0}{Z}{Z}{0}\right)=\mattwo{0}{Dw(X)(Z)}{Dw(X)(Z)}{0}.
\]
\end{proposition}

\begin{proof}[\bf{Proof}]
Note that
\[
V^*\mattwo{X+tZ}{0}{0}{X-tZ}V=\mattwo{X}{tZ}{tZ}{X},\quad\mbox{where}\quad t\in\BR,\ V=\frac{1}{\sqrt{2}}\mattwo{I}{I}{I}{-I}.
\]
Since $\cD$ is open and $X\in\cD$, for small $t$ both $2 \times 2$ block matrices are in $\cD$, and we have
\begin{align*}
&w\left(\mattwo{X}{tZ}{tZ}{X}\right)
=w\left(V^*\mattwo{X+tZ}{0}{0}{X-tZ}V\right)\\
&\qquad=V^*\mattwo{w(X+tZ)}{0}{0}{w(X-tZ)}V\\
&\qquad=\frac{1}{2}\mattwo{w(X+tZ)+w(X-tZ)}{w(X+tZ)-w(X-tZ)}{w(X+tZ)-w(X-tZ)}{w(X+tZ)+w(X-tZ)}.
\end{align*}
Using this formula we obtain
\begin{align*}
&Dw\left(\mattwo{X}{0}{0}{X}\right)\left(\mattwo{0}{Z}{Z}{0}\right)
=\lim_{t\to 0}\frac{w\left(\mattwo{X}{tZ}{tZ}{X}\right)-w\left(\mattwo{X}{0}{0}{X}\right)}{t}\\[.2cm]
&\quad=\frac{1}{2} \lim_{t\to0}  \mat{cc}{\frac{w(X+tZ)+w(X-tZ)-2w(X)}{t}&\frac{w(X+tZ)-w(X-tZ)}{t}\\[.2cm] \frac{w(X+tZ)-w(X-tZ)}{t}&\frac{w(X+tZ)+w(X-tZ)-2w(X)}{t}}\\[.2cm]
&\quad=\frac{1}{2}\lim_{t\to 0} \mat{cc}
{\frac{w(X+tZ)-w(X)}{t}-\frac{w(X-tZ)-w(X)}{-t} & \frac{w(X+tZ)-w(X)}{t}+\frac{w(X-tZ)-w(X)}{-t}\\[.2cm]
\frac{w(X+tZ)-w(X)}{t}+\frac{w(X-tZ)-w(X)}{-t} & \frac{w(X+tZ)-w(X)}{t}-\frac{w(X-tZ)-w(X)}{-t}}\\[.2cm]
&\quad=\frac{1}{2}\mattwo{Dw(X)(Z)-Dw(X)(Z)}{Dw(X)(Z)+Dw(X)(Z)}{Dw(X)(Z)+Dw(X)(Z)}{Dw(X)(Z)-Dw(X)(Z)}\\
&\quad=\mattwo{0}{Dw(X)(Z)}{Dw(X)(Z)}{0}.\qedhere
\end{align*}	
\end{proof}

\section{Real and complex part of a nc function}\label{S:RealComplex}

Throughout this section, let $f$ be a nc function with domain $\cD_f\subset \BC^d_\tu{nc}$. As in the introduction,
we define the real and imaginary parts of $f$ as
\begin{equation}\label{uvDef1}
\begin{aligned}
& u:\cD_u\to\cH_\tu{nc},\quad
u:\cD_v\to\cH_\tu{nc},\\
&\mbox{with } \cD_u=\cD_v=\cD:=\coprod_{n=1}^\infty \{(A,B)\colon A,B\in\cH_n^d,\ A+iB\in\cD_f\}\subset \cH_\tu{nc}^{2d},
\end{aligned}
\end{equation}
with $u$ and $v$ defined for $(A,B)\in\cD$ by
\begin{equation}\label{uvDef2}
\begin{aligned}
u(A,B)&:=\re f(A+iB)=\frac{1}{2}(f(A+iB)+f(A+iB)^\ast),\\
v(A,B)&:=\im f(A+iB)=\frac{1}{2i}(f(A+iB)-f(A+iB)^\ast).
\end{aligned}
\end{equation}
In particular, $u$, $v$ and $f$ satisfy \eqref{fdec}. The following theorem is the main result of this section.

\begin{theorem}\label{T:RCpart}
Let $f$ be a G-differentiable nc function defined on an open nc set $\cD_f\subset \BC^d_\tu{nc}$ and define $u$ and $v$ as in \eqref{uvDef1} and \eqref{uvDef2}. Then $u$ and $v$ are G-differentiable real nc functions, whose G-derivatives at $(A,B)\in\cD_{n}$ in direction $Z=(Z_1,Z_2)\in\cH_{n}^{2d}$, for any $n$, are given by
\begin{equation}\label{DuDv}
\begin{aligned}
Du(A,B)(Z_1,Z_2)&=\re Df(A+iB)(Z_1+iZ_2),\\
Dv(A,B)(Z_1,Z_2)&=\im Df(A+iB)(Z_1+iZ_2),
\end{aligned}
\end{equation}
and $Du$ and $Dv$ satisfy the nc Cauchy-Riemann equations:
\begin{equation}\label{CR-nec}
Du(A,B)(Z_1,Z_2)=Dv(A,B)(-Z_2,Z_1),\ (A,B)\in \cD_{n},\, Z_1,Z_2\in\cH_n,\, n\in\BN.
\end{equation}
Finally, if $f$ is F-differentiable, then $u$ and $v$ are F-differentiable as well.
\end{theorem}

In order to prove this result we first prove a lemma that will also be useful in the sequel. The result may be well-known, but we could not find it in the literature, hence we add a proof for completeness.

\begin{lemma}\label{L:normcomp}
For $Z=Z_1+iZ_2 \in \BC_\tu{nc}^d$ with $Z_1,Z_2\in\cH_n^d$ we have
\begin{equation}\label{NormIneqs}
\|(Z_1,Z_2)\|_{\cH_n^{2d}} \leq \|Z_1+iZ_2\|_{(\BC^{n\times n})^d}\leq 2 \|(Z_1,Z_2)\|_{\cH_n^{2d}}.
\end{equation}
\end{lemma}

\begin{proof}[\bf Proof]
Set $\de=\|Z_1+iZ_2\|=\|Z\|$, $\rho=\|(Z_1,Z_2)\|=\max\{\|Z_1\|,\|Z_2\|\}$. Then
\[
\de^2 I_n\geq Z^*Z =Z_1^2+Z_2^2 +[iZ_1,Z_2] \ands
\de^2 I_n\geq ZZ^* =Z_1^2+Z_2^2 -[iZ_1,Z_2].
\]
Here $[T_1,T_2]$ denotes the commutator of the square matrices $T_1$, $T_2$, i.e., $[T_1,T_2]=T_1T_2-T_2T_1$, which is applied entrywise in case $T_1$ and $T_2$ are tuples of matrices of the same size. Taking the average of the above two inequalities gives
\[
\de^2 I_n\geq Z_1^2+Z_2^2.
\]
Hence $Z_j^2\leq \de^2 I_n$, or equivalently, $\|Z_j\|\leq \de$ for both $j=1,2$. Therefore, we have $\|(Z_1,Z_2)\|\leq \de=\|Z_1+iZ_2\|$. For the second inequality, note that $Z_j^2\leq \rho^2 I_n $ for $j=1,2$.  Also, we have
\[
\|[i Z_1,Z_2]\|=\|Z_1Z_2-Z_2Z_1\|\leq 2\|Z_1\|\|Z_2\|\leq 2\rho^2.
\]
This implies $-2\rho^2 I_n \leq [iZ_1,Z_2] \leq 2 \rho^2 I_n$, since $[iZ_1,Z_2]\in\cH_n$. We then obtain
\[
0\leq Z^*Z=Z_1^2+Z_2^2 +[iZ_1,Z_2]\leq 4 \rho^2 I_n,
\]
so that $\|Z\|\leq 2\rho=2 \|(Z_1,Z_2)\|$.
\end{proof}

Since the inequalities in \eqref{NormIneq} provide a comparison between the norms in $\BC_\tu{nc}^d$ and $\cH_\tu{nc}^{2d}$, the following corollary is immediate.

\begin{corollary}
The nc set $\cD_f$ is open if and only if $\cD$ is open.
\end{corollary}

By applying the inequalities of Lemma \ref{L:normcomp} to both the denominator and numerator, we obtain the following corollary.

\begin{corollary}\label{C:normIneq}
Let $Z_1,Z_2\in\cH_n^d$ and $T_1,T_2\in\cH_m^k$. Then
\begin{equation}\label{NormIneq}
\frac{1}{2}\,\frac{\|(T_1,T_2)\|}{\|(Z_1,Z_2)\|}\leq \frac{\|T_1+iT_2\|}{\|Z_1+iZ_2\|} \leq 2 \frac{\|(T_1,T_2)\|}{\|(Z_1,Z_2)\|}.\medskip
\end{equation}
\end{corollary}

\begin{proof}[\bf Proof of Theorem \ref{T:RCpart}]
The proof is divided into four parts.\smallskip

\paragraph{\bf Part 1: $u$ and $v$ are real nc functions}
It is straightforward to check that $u$ and $v$ are graded and respect direct sums, since $f$ has these properties. Clearly $\cD$ is contained in $\cH_\tu{nc}^{2d}$. It remains to verify that $u$ and $v$ respect unitary equivalence. Let $(A,B)\in\cD_{n}$ and $U\in\BC^{n\times n}$ unitary so that $(UAU^*,UBU^*)\in\cD_{n}$. Set $X=A+iB \in\cD_f$. By definition of $\cD$ we have $UXU^*\in\cD_f$, and since $f$ respects similarities, and hence unitary equivalence, we have
\[
f(UXU^*) = U f(X) U^*.
\]
The left hand side specifies to
\[
f(UXU^*)=f(U A U^\ast+iU B U^\ast)=u(U A U^\ast,U B U^\ast)+iv(U A U^\ast,U B U^\ast),
\]
while on the right hand side we get
\[
U f(X) U^*=U f(A+iB) U^\ast=U u(A,B) U^\ast+iU v(A,B) U^\ast.
\]
Since the values of $u$ and $v$ are Hermitian and $\cH_n$ is closed under unitary equivalence, it follows that
\[
u(U A U^\ast,U B U^\ast)=U u(A,B) U^\ast \ands
v(U A U^\ast,U B U^\ast)=U v(A,B) U^\ast.
\]
Hence, $u$ and $v$ respect unitary equivalence.\smallskip

\paragraph{\bf Part 2: Proof of \eqref{DuDv}} Let $X=A+iB\in\cD_{f,n}$, $Z=Z_1+iZ_2\in (\BC^{n\times n})^d$ with $A,B,Z_1,Z_2\in\cH_\tu{nc}$. Assume $f$ is G-differentiable at X in direction $Z$. In this part we show that $u$ and $v$ are G-differentiable at $(A,B)$ in the direction $(Z_1,Z_2)$ and that their G-derivatives satisfy
\begin{equation}\label{DuDv2}
\begin{aligned}
Df(A+iB)(Z_1+i Z_2)&= Du(A,B)(Z_1,Z_2) + i Dv(A,B)(Z_1,Z_2).
\end{aligned}
\end{equation}
This proves \eqref{DuDv} and shows that $u$ and $v$ are G-differentiable in case $f$ is G-differentiable.

To see that our claim holds, note that for $0\neq t\in\BR$ we have
\begin{align*}
&\frac{f(X+tZ)-f(X)}{t}=\frac{f(A+iB+t(Z_1+iZ_2))-f(A+iB)}{t}\\
&\qquad=\frac{u(A+tZ_1,B+tZ_2)+iv(A+tZ_1,B+tZ_2)-u(A,B)-iv(A,B)}{t}\\
&\qquad=\frac{u(A+tZ_1,B+tZ_2)-u(A,B)}{t}+i\frac{v(A+tZ_1,B+tZ_2)-v(A,B)}{t}.
\end{align*}
The result follows by letting $t$ go to 0, and noting that in the right most side of the above identities the limits of the real and imaginary parts are independent.\smallskip

\paragraph{\bf Part 3: Cauchy-Riemann equations}
The proof follows along the same lines as the classical complex analysis proof. For $X=A+iB$, $Z=Z_1+iZ_2$ and $h\in\BR$ we have
\begin{align*}
&f(X+ihZ)-f(X)
=f(A+iB+ih(Z_1+iZ_2)-f(A+iB)\\
&\qquad =f(A-hZ_2+i(B+hZ_1))-f(A+iB)\\
&\qquad =u(A-hZ_2,B+hZ_1)+iv(A-hZ_2,B+hZ_1)-u(A,B)-iv(A,B)\\
&\qquad =u(A-hZ_2,B+hZ_1)-u(A,B)+i(v(A-hZ_2,B+hZ_1)-v(A,B)).
\end{align*}
Dividing by $ih$ and taking $h\to 0$ we obtain
\begin{align*}
D f (X)(Z)&=\lim_{h\to 0}\frac{f(X+ihZ)-f(X)}{ih}\\
&=\lim_{h\to 0}\frac{v(A-hZ_2,B+hZ_1)-v(A,B)}{h}+\\
&\qquad\qquad-i\lim_{h\to 0}\frac{u(A-hZ_2,B+hZ_1)-u(A,B)}{h}\\
&=Dv(A,B)(-Z_2,Z_1)-iDu(A,B)(-Z_2,Z_1).
\end{align*}
Comparing with \eqref{DuDv2} provides the desired equations.\smallskip

\paragraph{\bf Part 4: F-differentiability}
Assume $f$ is F-differentiable. This implies that $f$ is G-differentiable and hence $u$ and $v$ are G-differentiable, by Part 2. Since $Df$ is $\BC$-linear in the directional variable, it is clear from \eqref{DuDv} that $Du$ and $Dv$ are $\BR$-linear in the directional variable. Now let $X=A+iB$ with $(A,B)\in\cD_{n}$ and  $Z=Z_1+ i Z_2$ with $Z_1,Z_2\in\cH_n^d$. Then
\begin{align*}
&f(X+Z)-f(X)-Df(X)(Z)=\\
&\qquad= u(A+tZ_1,B+tZ_2)-u(A,B)-Du(A,B)(Z_1,Z_2) +\\
 & \qquad \qquad + i(v(A+tZ_1,B+tZ_2)-v(A,B)-Dv(A,B)(Z_1,Z_2)).
\end{align*}
Now apply Corollary \ref{C:normIneq} with $Z_1$ and $Z_2$ as above and
\begin{equation}\label{T1T2}
\begin{aligned}
T_1&=u(A+tZ_1,B+tZ_2)-u(A,B)-Du(A,B)(Z_1,Z_2),\\
T_2&=v(A+tZ_1,B+tZ_2)-v(A,B)-Dv(A,B)(Z_1,Z_2),
\end{aligned}
\end{equation}
and note that $Z\to 0$ if and only if $(Z_1,Z_2)\to 0$, by Lemma \ref{L:normcomp}. It then follows that
\begin{equation}\label{Fdiff}
 \lim_{\|Z\|\to 0} \frac{\|f(X+Z)-F(X)-Df(X)(Z)\|}{\|Z\|}=0
\end{equation}
holds if and only if
\begin{align*}
 \lim_{\|(Z_1,Z_2)\|\to 0} \frac{\|u(A+tZ_1,B+tZ_2)-u(A,B)-Du(A,B)(Z_1,Z_2)\|}{\|(Z_1,Z_2)\|}=0
\end{align*}
and
\begin{align*}
 \lim_{\|(Z_1,Z_2)\|\to 0} \frac{\|v(A+tZ_1,B+tZ_2)-v(A,B)-Dv(A,B)(Z_1,Z_2)\|}{\|(Z_1,Z_2)\|}=0.
\end{align*}
In particular, since \eqref{Fdiff} holds, and $(A,B)\in\cD_{n}$ and $Z_1,Z_2\in\cH_n^d$ were chosen arbitrarily, it follows that $u$ and $v$ are F-differentiable.
\end{proof}

The fact that the G-derivative of a G-differentiable nc function on a complex-open domain (and hence right-admissible) can be computed algebraically, via block upper triangular matrices, provides additional structure for its real and imaginary parts, which enables us to compute their G-derivatives algebraically as well.

\begin{proposition}\label{P:uvRespDiag}
Let $f$ be a nc function defined on an open nc set $\cD_f\subset\BC_\tu{nc}^d$ and define $u$ and $v$ as in \eqref{uvDef1}--\eqref{uvDef2}. Let $X=A+iB\in \cD_{f,n}$ and $Y=C+iD\in\cD_{f,m}$ and $Z\in(\BC^{n\times m})^d$ such that $\sbm{X&Z\\ 0 & Y}\in\cD_{f,n+m}$. Then
\begin{equation}\label{UTdec}
\left(\mat{cc}{A&\frac{1}{2}Z \\ \frac{1}{2}Z^* & C}, \mat{cc}{B&-\frac{i}{2}Z \\ \frac{i}{2}Z^* & D} \right)\in\cD
\end{equation}
and there exist $T_{X,Y,1},T_{X,Y,2}\in (\BC^{n\times m})^{2d}$ so that
\begin{equation}\label{DiagResp}
\begin{aligned}
u\left(\mat{cc}{A&\frac{1}{2}Z \\ \frac{1}{2}Z^* & C}, \mat{cc}{B&-\frac{i}{2}Z \\ \frac{i}{2}Z^* & D} \right)
&=\mat{cc}{u(A,B) & T_{X,Y,1}\\ T_{X,Y,1}^* & u(C,D)},\\
v\left(\mat{cc}{A&\frac{1}{2}Z \\ \frac{1}{2}Z^* & C}, \mat{cc}{B&-\frac{i}{2}Z \\ \frac{i}{2}Z^* & D} \right)
&=\mat{cc}{v(A,B) & T_{X,Y,2}\\ T_{X,Y,2}^* & v(C,D)}.
\end{aligned}
\end{equation}
Moreover, if $X=Y$, $Z=Z_1+ iZ_2$ with $Z_1,Z_2\in\cH_n^d$ and $f$ is locally bounded on slices, then
\[
Du(A,B)(Z_1,Z_2)= T_{X,Y,1}+ T_{X,Y,1}^*,\quad
Dv(A,B)(Z_1,Z_2)=T_{X,Y,2} + T_{X,Y,2}^*.
\]
\end{proposition}

\begin{proof}[\bf Proof]
The decomposition
\begin{equation}\label{Dec1}
\mat{cc}{X&Z\\0&Y}=\mat{cc}{A+iB& Z_1+iZ_2 \\ 0 &C+iD} =\mattwo{A}{\frac{1}{2}Z}{\frac{1}{2}Z^\ast}{C}+i\mattwo{B}{-\frac{i}{2}Z}{\frac{i}{2}Z^\ast}{D},
\end{equation}
together with $\sbm{X&Z\\ 0 & Y}\in\cD_{f,n+m}$ yields \eqref{UTdec}. Since $f$ is a nc function, we have
\[
f\left(\mattwo{X}{Z}{0}{Y}\right)=\mattwo{f(X)}{\De f(X,Y)(Z)}{0}{f(Y)},
\]
with $\De f(X,Y)(Z)$ the right nc difference-differential operator applied to $f$, at the point $(X,Y)$ and direction $Z$. Note that
\begin{align*}
\mattwo{f(X)}{\De f(X,Y)(Z)}{0}{f(Y)}
&=\mattwo{\frac{1}{2}(f(X)+f(X)^\ast)}{\frac{1}{2}\De f(X,Y)(Z)}{\frac{1}{2}\De f(X,Y)(Z)^\ast}{\frac{1}{2}(f(Y)+f(Y)^\ast)}+\\
&\qquad\qquad+i\mattwo{-\frac{i}{2}(f(X)-f(X)^\ast)}{-\frac{i}{2}\De f(X,Y)(Z)}{\frac{i}{2}\De f(X,Y)(Z)^\ast}{-\frac{i}{2}(f(Y)-f(Y)^\ast)}\\
&=\mattwo{u(A,B)}{\frac{1}{2}\De f(X,Y)(Z)}{\frac{1}{2}\De f(X,Y)(Z)^\ast}{u(C,D)}+\\
&\qquad\qquad+i\mattwo{v(A,B)}{-\frac{i}{2}\De f(X,Y)(Z)}{\frac{i}{2}\De f(X,Y)(Z)^\ast}{v(C,D)}.
\end{align*}
This formula for $f(\sbm{X&Z\\ 0 & Y})$ together with \eqref{Dec1} proves \eqref{DiagResp}, where we take $T_{X,Y,1}=\frac{1}{2}\De f(X,Y)(Z)$ and $T_{X,Y,2}=-\frac{1}{2}\De f(X,Y)(Z)$.

Now assume $X=Y$ and $f$ is locally bounded on slices. Then $f$ is G-differentiable and $\De f(X,Y)(Z)=D f(X)(Z)$. It now follows by Theorem \ref{T:RCpart} that
\[
T_{X,Y,1}+T_{X,Y,1}^*= \re D f(X)(Z) = D u (A,B)(Z_1,Z_2),
\]
and, similarly, $D v (A,B)(Z_1,Z_2)=T_{X,Y,2}+ T_{X,Y,2}^*$.
\end{proof}

Not all real nc functions ``respect diagonals'' as in \eqref{DiagResp}. Also, one may wonder whether \eqref{DiagResp} in some form extends beyond points of the form \eqref{UTdec} in case $u$ and $v$ are the real and imaginary parts of a nc function. This is also not the case in general. We illustrate this in the following example.

\begin{example}
Consider the following three real nc functions
\[
u(A,B)=A^2-B^2,\quad v(A,B)=AB+BA,\quad w(A,B)=A^2\quad ((A,B)\in\cH_\tu{nc}^2).
\]
Then $u$ and $v$ are the real and imaginary part of the nc function $f(X)=X^2$. For an arbitrary 2 $\times$ 2 block point
\[
(E,F):=\left(\mat{cc}{A&Z_1 \\ Z_1^* & C}, \mat{cc}{B&Z_2 \\ Z_2^* & D} \right)\in \cH_\tu{nc}^2
\]
we obtain:
\begin{align*}
u (E,F)
& = \mat{cc}{A^2-B^2 +Z_1Z_1^*- Z_2Z_2^* & AZ_1-B Z_2+ Z_1 C - Z_2 D  \\
Z_1^* A-Z_2^*B+ CZ_1^* -D Z_2^* & C^2-D^2 +Z_1^*Z_1- Z_2^*Z_2},\\
v(E,F)
& =\mat{cc}{AB + BA + Z_1 Z_2^* + Z_2Z _1^* & BZ_1 + AZ_2 +  Z_2C +  Z_1D \\
Z_1^* B + Z_2^*A + C Z_2^* + D Z_1^* & CD +DC + Z_1^* Z_2 + Z_2^* Z_1},\\
w(E,F)
& = \mat{cc}{A^2 +Z_1Z_1^* & AZ_1 + Z_1 C  \\ Z_1^* A+ CZ_1^*  & C^2 +Z_1Z_1^*}.
\end{align*}
It follows that $u(E,F)=\sbm{u(A,B) & *\\ * & u(C,D)}$ holds if and only if
\begin{equation}\label{Con1}
Z_1 Z_1^*=Z_2Z_2^* \ands Z_1^* Z_1=Z_2^* Z_2,
\end{equation}
while $v(E,F)=\sbm{v(A,B) & *\\ * & v(C,D)}$ holds if and only if
\begin{equation}\label{Con2}
Z_1 Z_2^*=-Z_2Z_1^* \ands Z_1^* Z_2=-Z_2^*Z_1.
\end{equation}
Both conditions are true in case $Z_2=\pm iZ_1$. Conversely, these conditions on $Z_1$ and $Z_2$ together imply $Z_2=-iZ_1$, but, in general, neither implies $Z_2=\pm iZ_1$ by itself. Indeed, the identities in \eqref{Con1} imply that the kernels and co-kernels of $Z_1$ and $Z_2$ coincide, so that we can reduce to the case where $Z_1$ and $Z_2$ are invertible. In that case, by Douglas' Lemma,  \eqref{Con1} is equivalent to the existence of unitary matrices $U$ and $V$ so that $Z_1=U Z_2=Z_2 V$. Assume $U$ and $V$ are like this, and $Z_1,Z_2$ invertible. Then \eqref{Con2} implies
\[
Z_2Z_2^*U Z_2Z_2^* = Z_2Z_2^* Z_1Z_2^* =- Z_2Z_1^* Z_2Z_2^* =-Z_2Z_2^*U^* Z_2Z_2^*.
\]
However, $Z_2$ is invertible, hence $Z_2Z_2^*$ is invertible. Thus we find that $U=-U^*$, which implies $U=\pm i I$. Hence $Z_1=\pm i Z_2$.

On the other hand, we have $w(E,F)=\sbm{w(A,B) & *\\ * & w(C,D)}$  precisely when $Z_1=0$. Hence \eqref{DiagResp} holds with $u$ or $v$ replaced by $w$ if and only if $Z=0$, which is true for any real nc function.
\end{example}

%\newpage
%%%%%%%%%%%%%%%%%%%%%%%%%%%%%%%%%%%%%%%%%%%%%%%%%%
%%%%%%%%%%%%%%%%%%%%%%%%%%%%%%%%%%%%%%%%%%%%%%%%%%
\section{Cauchy-Riemann equations: Sufficiency}\label{S:CR}

In this section we prove Theorem \ref{T:CR}. Throughout, let
\begin{equation}\label{uv}
u:\cD_u\to\cH_\tu{nc}^{2d} \ands v:\cD_v\to\cH_\tu{nc}^{2d}
\end{equation}
be real nc functions. For notational convenience we introduce the nc set
\[
\cD:=\cD_u \cap \cD_v.
\]
Now we define $f$ on $\cD_f:=\{A+i B \colon (A,B)\in \cD\}$ by
\begin{equation}\label{fdef}
f(A+iB)=u(A,B)+i v(A,B)\quad (A+iB\in\cD_f).
\end{equation}
It is easy to see that $f$ is graded, respects direct sums as well as unitary equivalence, since $u$ and $v$ have these properties. However, it is not necessarily the case that $f$ respects similarities, despite the fact that $u$ and $v$ do. The following proposition sums up the properties that $f$ has without further assumptions on $u$ and $v$ (except G-differentiability in the last part). The claims follow directly from \eqref{fdef}, hence we omit the proof.

\begin{proposition}\label{P:uvtof}
Let $u$ and $v$ be real nc functions as in \eqref{uv} and define $f$ as in \eqref{fdef}. Then $f$ is graded, respects direct sums and respects unitary equivalence. Moreover, in case $X=A+iB$ with $(A,B)\in \cD_{n}$, $Z=Z_1+i Z_2$ with $Z_1,Z_2\in\cH_n$, for any $n\in\BN$, and $u$ and $v$ are G-differentiable at $(A,B)$ in direction $(Z_1,Z_2)$, then
\begin{equation}\label{RealDiff}
\lim_{\BR\ni t\to 0}
\frac{f(X+tZ)-f(X)}{t}=Du(A,B)(Z_1,Z_2) + i Dv(A,B)(Z_1,Z_2).
\end{equation}
\end{proposition}

\begin{remark}
Without additional assumptions on $u$ and $v$ it is possible to prove something slightly stronger than the fact that $f$ respects unitary similarity. If $X=A+iB\in\cD_{f,n}$ and $S\in \BC^{n\times n}$ is invertible are such that $C:=S A S^{-1}$ and $D:= S BS^{-1}$ are in $\cH_n^d$, then it still follows easily that $f(SXS^{-1})=S f(X)S^{-1}$, using the fact that $u$ and $v$ respect similarity. Note that in this case $(A,B)$ and $(C,D)$ are not only similar via $S$, but also unitarily equivalent via the unitary matrix in the polar decomposition of $S$, cf., Remark \ref{R:iii'}. In general, of course, it will not be the case that $C$ and $D$ are Hermitian.
\end{remark}

To prove, under the conditions of Theorem \ref{T:CR}, that $f$ respects similarity, and hence is a nc function, we will use Lemma 2.3 of \cite{PTD17}. To apply this lemma, we need to prove that $f$ has the following two properties:
\begin{itemize}
  \item[(i)] $f$ is F-differentiable;
  \item[(ii)] the following identity holds
  \begin{equation}\label{ComForm0}
Df(X)([T,X])=[T,f(X)],\quad X\in\cD_{f,n},\ T\in\BC^{n\times n},\ n=1,2,\ldots.
\end{equation}
\end{itemize}
As before, $[S,Q]$ denotes the commutator of square matrices $S,Q$ of the same size, applied entrywise in case $S$ and $Q$ are tuples of matrices. In case only one of $S$ and $Q$ is a tuple, then the other one is identified with a tuple of the same length and the given matrix in each entry.  Note that if $S$ and $Q$ are Hermitian, then $[S,Q]$ is skew-Hermitian, and hence $[iS,Q]=i[S,Q]$ is Hermitian.

To achieve more than in Proposition \ref{P:uvtof} we require the nc Cauchy-Riemann equations \eqref{CR-intro} which, for convenience, we recall here: For $n=1,2,\ldots$
\begin{equation}\label{CR2}
Du(A,B)(Z_1,Z_2)=Dv(A,B)(-Z_2,Z_1),\ (A,B)\in \cD_{n},\ Z_1,Z_2\in\cH_n.
\end{equation}

From Proposition \ref{P:uvtof} it is clear what the G-derivative of $f$ should be in case $f$ is F-differentiable. For $X=A+iB\in\cD_{f,n}$ and $Z_1+iZ_2\in(\BC^{n\times n})^d$ we define
\begin{align}\label{DefTilD}
\wtil{D}f(A+iB)(Z_1+iZ_2) & := Du(A,B)(Z_1,Z_2) + i Dv(A,B)(Z_1,Z_2),
\end{align}
provided the G-derivatives of $u$ and $v$ exist in $(A,B)$. As a first step we show that $\wtil{D}f(X)(Z)$ is linear in $Z$.

\begin{lemma}\label{L:Hom}
Let $u$ and $v$ be  G-differentiable, real nc functions that satisfy the nc Cauchy-Riemann equations \eqref{CR2}. Then the map $\wtil{D}f(X)(Z)$ defined in \eqref{DefTilD} is linear in the directional variable $Z$.
\end{lemma}

\begin{proof}[\bf Proof] The maps $D_u$ and $D_v$ are $\BR$-linear in the directional variable. Hence $\wtil{D}f$ is additive and $\BR$-homogeneous in the directional variable. Write $z\in\BC$ as $z=r e^{i \theta}$ with $r\geq 0$ and $\theta\in[0,2\pi]$. Note that
\begin{align*}
e^{i\theta} Z & =(\cos\theta + i \sin \theta)(Z_1+i Z_2)
=( Z_1 \cos \theta -  Z_2 \sin \theta) + i( Z_1 \sin\theta +  Z_2\cos\theta).
\end{align*}
Set $Z_{1,\theta}:= Z_1 \cos \theta -  Z_2 \sin \theta$ and $Z_{2,\theta}:=Z_1 \sin\theta +  Z_2\cos\theta$. It follows that
\begin{align}
\wtil{D}f(X)(zZ)
&=Du(A,B)(rZ_{1,\theta},rZ_{2,\theta}) + i Dv(A,B)(rZ_{1,\theta},rZ_{2,\theta})\notag \\
&=r(Du(A,B)(Z_{1,\theta},Z_{2,\theta}) + i Dv(A,B)(Z_{1,\theta},Z_{2,\theta})).\label{Id1}
\end{align}

Using that G-derivatives $Du$ and $Dv$ are $\BR$-linear in the the directional variables together with the Cauchy-Riemann equations \eqref{CR2} yields
\begin{align*}
Du(A,B)(Z_{1,\theta},Z_{2,\theta}) & =
\cos\theta Du(A,B)(Z_1,Z_2) + \sin\theta Du(A,B)(-Z_2,Z_1) \\
&=\cos\theta Du(A,B)(Z_1,Z_2) - \sin\theta Dv(A,B)(Z_1,Z_2).
\end{align*}
Similarly, we have
\begin{align*}
Dv(A,B)(Z_{1,\theta},Z_{2,\theta}) & =
\cos\theta Dv(A,B)(Z_1,Z_2) + \sin\theta Dv(A,B)(-Z_2,Z_1) \\
&=\cos\theta Dv(A,B)(Z_1,Z_2) + \sin\theta Du(A,B)(Z_1,Z_2).
\end{align*}
Combining these formulas shows
\begin{align}
&Du(A,B)(Z_{1,\theta},Z_{2,\theta}) + i Dv(A,B)(Z_{1,\theta},Z_{2,\theta})
=\notag \\
&\quad=(\cos \theta + i\sin \theta) Du(A,B)(Z_1,Z_2)
+((\cos \theta + i\sin \theta)) i Dv(A,B)(Z_1,Z_2)\notag \\
&\quad=e^{i\theta}(Du(A,B)(Z_1,Z_2) +i Dv(A,B)(Z_1,Z_2)).
\end{align}
Together with \eqref{Id1} this yields
\[
\wtil{D}f(X)(zZ)=z(Du(A,B)(Z_1,Z_2) +i Dv(A,B)(Z_1,Z_2)),
\]
so that $\wtil{D}$ is $\BC$-homogeneous in the directional variable, and hence $\BC$-linear.
\end{proof}

With linearity out of the way, it is straightforward to prove $f$ is F-differentiable in case $u$ and $v$ are F-differentiable.

\begin{lemma}\label{L:compdiff}
Let $u$ and $v$ be  F-differentiable, real nc functions that satisfy the nc Cauchy-Riemann equations \eqref{CR2}. Then $f$ defined by \eqref{fdef} is F-differentiable with G-derivative given by $Df(X)(Z)=\wtil{D}f(X)(Z)$ as in \eqref{DefTilD}.
\end{lemma}

\begin{proof}[\bf Proof] The proof is similar to the last part of the proof of Theorem \ref{T:RCpart}.
Since $u$ and $v$ are F-differentiable, they are G-differentiable, and thus $\wtil{D}f$ is $\BC$-linear in the directional variable. To see that $f$ is F-differentiable,  note that for $X=A+iB\in\cD_{f,n}$ and $Z=Z_1+iZ_2$, $Z_1,Z_2\in\cH_n^d$, we have
\begin{align*}
&f(X+Z)-f(X)-\wtil{D}f(X)(Z)=\\
&\qquad=(u(A+Z_1,B+Z_2)-u(A,B)-Du(A,B)(Z_1,Z_2))+\\
&\qquad \qquad +i(v(A+Z_1,B+Z_2)-v(A,B)-Dv(A,B)(Z_1,Z_2)).
\end{align*}
Using $T_1$ and $T_2$ as in \eqref{T1T2} the same argument applies, in the opposite direction, to conclude that F-differentiability of $u$ and $v$ implies F-differentiability of $f$.
\end{proof}

\begin{lemma}\label{L:ComForm}
Let $u$ and $v$ be  F-differentiable, real nc functions that satisfy the nc Cauchy-Riemann equations \eqref{CR2}.  Define $f$ as in \eqref{fdef}. Then \eqref{ComForm0} holds.
\end{lemma}

\begin{proof}[\bf Proof]
%By Lemma 2.3 in \cite{PTD17}, $u$ and $v$ being real nc functions, implies they satisfy
%\begin{align*}
%Du(A,B)([iR,(A,B)])&=[iR,u(A,B)],\\
% Dv(A,B)([iR,(A,B)])&=[iR,v(A,B)],
%\end{align*}
%for any $(A,B)\in\cD_{u,v}$ and any Hermitian $R$ of appropriate size.
Let $X=A+iB$ and $T=T_1+i T_2$. Then
\[
[T,X]=([iT_1,B]+[iT_2,A])+i([iT_1,-A]+[iT_2,B]).
\]
Set $Z_1=[iT_1,B]+[iT_2,A]$ and $Z_2=[iT_1,-A]+[iT_2,B]$. By Lemma \ref{L:compdiff} we obtain
\[
Df(X)([T,X])=Du(A,B)(Z_1,Z_2) + i Dv(A,B)(Z_1,Z_2).
\]
Note that
\begin{align*}
Du(A,B)(Z_1,Z_2)
& = Du(A,B)([iT_1,B]+[iT_2,A],[iT_1,-A]+[iT_2,B])\\
& = Du(A,B)([iT_2,A],[iT_2,B])
+ Du(A,B)([iT_1,B],[iT_1,-A])\\
& = Du(A,B)([iT_2,(A,B)])
+ Du(A,B)([iT_1,(B,-A)]).
\end{align*}
Applying the Cauchy-Riemann equations \eqref{CR2} to the second summand gives
\begin{align*}
Du(A,B)(Z_1,Z_2)
& = Du(A,B)([iT_2,(A,B)])
+ Dv(A,B)([iT_1,(A,B)]).
\end{align*}
Now use that Part (a) Lemma 2.3 of \cite{PTD17} applies to $u$ and $v$. This yields
\begin{align*}
Du(A,B)(Z_1,Z_2) & = [iT_2,u(A,B)] + [iT_1,v(A,B)].
\end{align*}
Similarly, for $Dv(A,B)(Z_1,Z_2)$ we get
\begin{align*}
Dv(A,B)(Z_1,Z_2)
& = Dv(A,B)([iT_2,(A,B)])
+ Dv(A,B)([iT_1,(B,-A)])\\
& = Dv(A,B)([iT_2,(A,B)])
+ Dv(A,B)([iT_1,(-A,-B)])\\
& = Dv(A,B)([iT_2,(A,B)])
- Du(A,B)([iT_1,(A,B)])\\
& = [iT_2, v(A,B)] - [iT_1, u(A,B)].
\end{align*}
Therefore, we have
\begin{align*}
&Df(X)([T,X])=\\
&\qquad\qquad=[iT_2,u(A,B)] + [iT_1,v(A,B)]
 + i ([iT_2, v(A,B)] - [iT_1, u(A,B)])\\
 &\qquad\qquad=[iT_2,u(A,B)]  - i[iT_1, u(A,B)]
  + [iT_1,v(A,B)] + i [iT_2, v(A,B)] \\
   &\qquad\qquad=[T_1+ iT_2,u(A,B)]   + [T_1 + i T_2,i v(A,B)]\\
   &\qquad\qquad=[T, u(A,B)+ i v(A,B) ]=[T,f(X)].\qedhere
\end{align*}
\end{proof}

\begin{proof}[\bf Proof of Theorem \ref{T:CR}]
The proof of this theorem is now straightforward. The fact that $f$ is graded and respects direct sums follows from Proposition \ref{P:uvtof}. Lemma \ref{L:compdiff} yields the F-differentiability of $f.$ Finally, from Lemma \ref{L:ComForm} we have that \eqref{ComForm0} holds and combining this with the fact that $f$ is F-differentiable we can apply Lemma 2.3 of \cite{PTD17} to conclude that $f$ respects similarities. Therefore, $f$ is a F-differentiable nc function.
\end{proof}

\begin{remark}\label{R:G-diff}
As pointed out in \cite{PTD17}, even in classical complex analysis, G-differen\-tiability of $u$ and $v$, i.e., existence of partial derivatives, together with the Cauchy-Riemann equations is not strong enough to prove analyticity of $f$. Continuity of the partial derivatives provides F-differentiability, which is strong enough; this corresponds to the approach taken in the present paper. The Looman-Menchoff theorem, cf., \cite[Page 199]{S64}, states that continuity of $f$, and hence of $u$ and $v$, is also sufficient. This in turn implies that $u$ and $v$ were F-differentiable from the start. As the proof of the Looman-Menchoff theorem requires the Baire category theorem and Lebesgue integration, it is not clear whether a similar relaxation of Theorem \ref{T:CR} can be achieved in the context considered here. In particular, the theory of integration of nc functions does not appear to be well developed so far. We are just aware of the paper \cite{PV18} on the nc Hardy space over the unitary matrices.
 \end{remark}

\paragraph{\bf Acknowledgments}
This work is based on research supported in part by the National Research Foundation of South Africa (NRF) and the DST-NRF Centre of Excellence in Mathematical and Statistical Sciences (CoE-MaSS). Any opinion, finding and conclusion or recommendation expressed in this material is that of the authors and the NRF and CoE-MaSS do not accept any liability in this regard.

%%%%%%%%%%%%%%%%%%%%%%%%%%%%%%%%%%%%%%%%%%%%%%%%%%%%%%%%%%%%%%%%%%%%%%%%
%%%%%%%%%%%%%%%%%%%%%%%%%%%%%%%%%%%%%%%%%%%%%%%%%%%%%%%%%%%%%%%%%%%%%%%%


\begin{thebibliography}{99}

%\bibitem{AMH08}
%A.H. Al-Mohy and N.J. Higham, Computing the Fr\'{e}chet derivative of the matrix exponential, with an application to condition number estimation, {\em SIAM J. Matrix Anal.\ Appl.} {\bf 30} (2008/09), 1639–-1657.

\bibitem{AH01}
K. Atkinson and W. Han, {\em Theoretical numerical analysis. A functional analysis framework}, Texts in Applied Mathematics {\bf 39}, Springer-Verlag, New York, 2001.

%\bibitem{B13}
%K. Bickel, Differentiating matrix functions, {\em Oper.\ Matrices} {\bf 7} (2013), 71–-90.

\bibitem{DGM75}
S.A.R. Disney, J.D. Gray, and S.A. Morris, Is a function that satisfies the Cauchy-Riemann equations necessarily analytic? {\em Austral.\ Math.\ Soc.\ Gaz.} {\bf 2} (1975), 67–81.

%\bibitem{AKMcC11}
%J.W. Helton, I. Klep, and S. McCullough, Proper analytic free maps, {\em, J. Funct. Anal.} {\bf 260} (2011), 1476-–1490.

%\bibitem{H08}
%N.J. Higham, {\em Functions of matrices. Theory and computation}, Society for Industrial and Applied Mathematics (SIAM), Philadelphia, PA, 2008.

\bibitem{JS18}
T. Jiang and H. Sendov, On differentiability of a class of orthogonally invariant functions on several operator variables, {\em Oper.\ Matrices} {\bf 12} (2018), 711–-721.

\bibitem{GM78}
J.D. Gray and S.A. Morris, When is a function that satisfies the Cauchy-Riemann equations analytic? {\em Amer.\ Math.\ Monthly} {\bf 85} (1978),  246–256.

\bibitem{HP57}
E. Hille and R.S. Phillips, {\em Functional analysis and semi-groups}, American Mathematical Society Colloquium Publications {\bf 31}, American Mathematical Society, Providence, R.I., 1957.

\bibitem{HJ13}
R.A. Horn and C.R. Johnson, {\em Matrix analysis. Second edition}, Cambridge University Press, Cambridge, 2013.

\bibitem{KVV14}
D.S. Kaliuzhnyi-Verbovetskyi and V. Vinnikov, {\em Foundations of free noncommutative function theory}, Mathematical Surveys and Monographs {\bf 199}, American Mathematical Society, Providence, RI, 2014.

\bibitem{P-Arx16}
M. P\'{a}lfia, L\"{o}wner's Theorem in several variables, preprint, arXiv:1405.5076.

\bibitem{P18}
J.E. Pascoe, The noncommutative Löwner theorem for matrix monotone functions over operator systems, {\em Linear Algebra Appl.} {\bf 541} (2018), 54–-59.

\bibitem{PTD17}
J.E. Pascoe and R. Tully-Doyle, Free Pick functions: representations, asymptotic behavior and matrix monotonicity in several noncommuting variables, {\em J. Funct.\ Anal.} {\bf 273} (2017), 283-–328.

%\bibitem{RW01}
%R. Rochberg and N. Weaver, Noncommutative complex analysis and Bargmann-Segal multipliers, {\em Proc.\ Amer.\ Math.\ Soc.} {\bf 129} (2001), 2679–-2687. PROBABLY NOT RELEVANT.

\bibitem{P13}
J.-P. Penot, {\em Calculus without derivatives}, Graduate Texts in Mathematics {\bf 266}, Springer, New York, 2013.

\bibitem{PV18}
M. Popa and V. Vinnikov, $H^2$ spaces of non-commutative functions, {\em Complex Anal.\ Oper.\ Theory} {\bf 12} (2018), 945-–967.

\bibitem{S64}
S. Saks, {\em Theory of the integral. Second revised edition}, Dover Publications, Inc., New York, 1964.

%\bibitem{S92}
%B.V. Shabat, {\em Introduction to complex analysis. Part II. Functions of several variables}, Translations of Mathematical Monographs {\bf 110}, American Mathematical Society, Providence, RI, 1992.

%\bibitem{T70}
%J.L. Taylor, The analytic-functional calculus for several commuting operators, {\em Acta Math.} {\bf 125} (1970), 1–-38.

%\bibitem{T72}
%J.L. Taylor, A general framework for a multi-operator functional calculus, {\em Advances in Math.} {\bf 9} (1972), 183–-252.

\bibitem{T73}
J.L. Taylor, Functions of several noncommuting variables, {\em Bull.\ Amer.\ Math.\ Soc.} {\bf 79} (1973), 1–-34.

\bibitem{Z46}
M.A. Zorn,  Derivatives and Fréchet differentials, {\em Bull.\ Amer.\ Math.\ Soc.} {\bf 52} (1946), 133–137.

\end{thebibliography}
\end{document}